\documentclass[11pt]{birkjour}
\usepackage[table,xcdraw]{xcolor}
\usepackage{tikz}
\usepackage{lipsum}
\usetikzlibrary{arrows.meta, positioning}
\usepackage{tikz-cd}
\usepackage{mathdots}
\usepackage{multicol}
\usepackage{color}
\usepackage{faktor}
\usepackage{textcomp}
\usepackage{tkz-graph}
\usepackage{tikz-cd}
\usetikzlibrary{calc}
\usepackage{tikz}
\usetikzlibrary{arrows}
\usepackage{stackrel}
\usepackage{faktor}
\usepackage{stackengine}
\usepackage{amssymb}
\usepackage[mathscr]{euscript}
\usepackage{comment}
\usepackage[skip=0.4\baselineskip]{caption}
\usepackage{hyperref}
\stackMath

 \newtheorem{thm}{Theorem}[section]
 \newtheorem{cor}[thm]{Corollary}
 \newtheorem{lem}[thm]{Lemma}
 
 \theoremstyle{definition}
 \newtheorem{defin}[thm]{Definition}
 \theoremstyle{remark}
 \newtheorem{rem}[thm]{Remark}
 
 \numberwithin{equation}{section}

\begin{document}

%
%
%
%
%
%
%
%
%
\title[On Some Properties of Matrices with Entries Defined by Products of $k$-Fibonacci and $k$-Lucas Numbers]
{On Some Properties of Matrices with Entries Defined by Products of $k$-Fibonacci and $k$-Lucas Numbers}
\author[PFFE]{Pedro Fernando Fern\'andez Espinosa}
\address{Semillero de Investigación ITENUA\\
	Escuela de Matemáticas y Estadistica  \\
	Universidad Pedagógica y Tecnológica de Colombia \\
	Carrera 18 con Calle 22\\
	Duitama - Boyacá\\
    ORCID ID: 0000-0002-2650-4536}
\email{pedro.fernandez01@uptc.edu.co}
\thanks{This work was completed with the support of Universidad Pedagógica y Tecnológica de Colombia.}
\author[MLAT]{Maritza Liliana Arciniegas Torres}
\address{Semillero de Investigación ITENUA\\
    Escuela de Matemáticas y Estadistica  \\
	Universidad Pedagógica y Tecnológica de Colombia \\
	Carrera 18 con Calle 22\\
	Duitama - Boyacá\\
    ORCID ID: 0009-0003-5189-9976}
\email{maritza.arciniegas@uptc.edu.co}

\author[CAAC]{Camilo Andrés Acevedo Cadena }
\address{Semillero de Investigación ITENUA\\
    Escuela de Matemáticas y Estadistica  \\
	Universidad Pedagógica y Tecnológica de Colombia \\
	Carrera 18 con Calle 22\\
	Duitama - Boyacá\\
    ORCID ID: 0009-0002-8366-4277}
\email{camilo.acevedo02@uptc.edu.co}

\date{January 10, 2026}
\begin{abstract}
In this paper, we study a structured family of matrices whose entries are given by products of $k$-Fibonacci and $k$-Lucas numbers. For this family, we obtain explicit and unified formulas for several classical matrix invariants, including the determinant, inverse, trace, and matrix powers, revealing nontrivial algebraic patterns induced by the underlying recurrence relations. In addition, we determine the spectral radius and the energy of the graphs naturally associated with these matrices. Finally, we establish connections between the resulting formulas and certain integer sequences recorded in the On-Line Encyclopedia of Integer Sequences (OEIS).
\end{abstract}
\maketitle
\section{Introduction}

\noindent Integer sequences play a fundamental role in many areas of mathematics and its applications. Within this vast landscape, certain sequences stand out due to their theoretical significance and practical relevance. Among the most notable examples are the Fibonacci numbers, the Lucas numbers, and their corresponding generalizations, which exhibit a wide range of remarkable properties and find applications across numerous fields of science and art. According to Koshy (see \cite{koshy}), the Fibonacci and Lucas sequences are among the most prominent and influential integer sequences. These sequences have fascinated both amateur and professional mathematicians for centuries and continue to attract considerable interest owing to their structural elegance, extensive applications, and frequent appearance in diverse and seemingly unrelated contexts. Consequently, they remain an active and fertile area of research in contemporary mathematics.
\par\smallskip

\noindent
In the context of linear algebra and matrix theory, matrices with integer entries have been extensively studied. Among the numerous contributions in this direction, works by M.~Janji\'{c} \cite{Janjic2010}, D.~Bozkurt \cite{Bozkurt2013}, T.~M.~Cronin \cite{Cronin2018}, J.~Rimas \cite{Rimas2007}, and other authors (see, for instance, \cite{BahsiSolak2015,BicknellJohnsonSpears1996}) focus on structured matrices, such as Toeplitz and Hessenberg matrices, with integer entries. These studies address a broad spectrum of topics, ranging from classical properties, including determinants and traces, to more recent aspects such as characteristic polynomials, eigenvalues, spectral radius, matrix norms, and applications to graph energy, among others.

\par\smallskip

\noindent
For the particular case of the Fibonacci and Lucas sequences, as well as their generalizations such as the $k$-Fibonacci sequence, research has primarily focused on the study of various classes of matrices whose entries involve these numbers. For instance, in \cite{LeeLeeShin1997} and \cite{LeeKim2003}, the authors introduce $k$-Fibonacci matrices and their generalizations and establish several results describing their fundamental properties. In particular, special attention is devoted to the behavior of their eigenvalues and to their connections with graph theory. Similarly, related investigations have been carried out for tridiagonal Fibonacci matrices by Cahill \cite{CahillNarayan2004}, for tridiagonal constructions by Goy \cite{GoyShattuck2020}, for pentadiagonal matrices by \.{I}pek \cite{IpekAri2014}, for further generalizations by Karaduman \cite{Karaduman2005,Karaduman2008} and Solak \cite{Solak2005,SolakBozkurt2005}, and more recently by Doumas \cite{Doumas2025}, where a Redheffer-type matrix with Fibonacci entries is introduced and its determinant and spectral properties are studied.
\par\smallskip

\noindent
In this work, we focus on the study of two families of matrices. The first family was introduced by Michel Bataille in Elementary Problem~B1360 of The Fibonacci Quarterly (see \cite{fibonacci2024}). The second family is defined by the authors as a generalization of the first one, whose entries are given by products of $k$-Fibonacci and $k$-Lucas numbers. For these matrices, we employ combinatorial and algebraic techniques to derive closed-form expressions for several classical matrix invariants, including the determinant, inverse, trace, and matrix powers. Furthermore, we investigate spectral properties such as the spectral radius and the energy of the graphs associated with these matrices.
\par\smallskip

\noindent
To achieve these results, in some cases we exploit the structural properties of the $k$-Fibonacci and $k$-Lucas numbers together with elementary tools from linear algebra. In other cases, we employ more advanced techniques, such as the Sherman-Morrison lemma, among others. Finally, and of equal importance, we establish connections between the derived formulas and several integer sequences recorded in the On-Line Encyclopedia of Integer Sequences (OEIS).
\par\smallskip

\noindent
This paper is organized as follows. In Section~\ref{preliminaries}, we recall the main notation and definitions concerning Fibonacci numbers, Lucas numbers, $k$-Fibonacci numbers, and $k$-Lucas numbers, as well as basic concepts from spectral theory.
\par\smallskip

\noindent
In Section~\ref{Generic}, we study classical properties such as the determinant, inverse, trace, and powers of the matrices introduced by Michel Bataille. In addition, we analyze the spectrum of these matrices and identify several integer sequences from the On-Line Encyclopedia of Integer Sequences (OEIS) that arise from these computations.
\par\smallskip

\noindent
Finally, in Section~\ref{General}, we introduce a new family of matrices that generalizes those studied in Section~\ref{Generic}. For this family, we present a suitable factorization that allows us to derive closed-form expressions for several classical matrix invariants, including the determinant, inverse, trace, and matrix powers. Moreover, we compute the spectral radius and the energy of the graphs associated with this family of matrices. As in Section~\ref{Generic}, we also investigate connections between the obtained formulas and certain integer sequences listed in the On-Line Encyclopedia of Integer Sequences (OEIS).

\section{Preliminaries}\label{preliminaries}
In this section, we recall the main definitions and notation that will be used throughout the paper (see, for example, \cite{koshy,Falcon2011,FalconPlaza2007,FalconPlaza2009}).
\par\smallskip

\noindent
The Fibonacci numbers $\{F_n\}_{n\geq 0}$ form the sequence $0,1,1,2,3,5,\ldots$, where each term is the sum of the two preceding ones, with initial conditions $F_0=0$ and $F_1=1$. As a generalization of this family, introduced by Falcon \cite{Falcon2011}, for any integer $k\geq 1$ the $k$-Fibonacci sequence $\{F_{k,n}\}_{n\in\mathbb{N}}$ is defined recursively by
\[
F_{k,0}=0,\quad F_{k,1}=1,\quad \text{and}\quad F_{k,n+1}=kF_{k,n}+F_{k,n-1},\qquad n\geq 1.
\]

\par\smallskip
\noindent
From this definition, the first values of the $k$-Fibonacci numbers are listed in Table~\ref{tab1}. These expressions allow one to compute any $k$-Fibonacci number by direct substitution.

\begin{table}[ht]
\centering
\[
\begin{aligned}
& F_{k,1} = 1, \\
& F_{k,2} = k, \\
& F_{k,3} = k^2 + 1, \\
& F_{k,4} = k^3 + 2k, \\
& F_{k,5} = k^4 + 3k^2 + 1, \\
& F_{k,6} = k^5 + 4k^3 + 3k, \\
& F_{k,7} = k^6 + 5k^4 + 6k^2 + 1, \\
& F_{k,8} = k^7 + 6k^5 + 10k^3 + 4k.
\end{aligned}
\]
\caption{First values of the $k$-Fibonacci numbers}
\label{tab1}
\end{table}

\noindent
Several $k$-Fibonacci sequences are indexed in the On-Line Encyclopedia of Integer Sequences \cite{OEIS} (OEIS). In particular,
\begin{itemize}
    \item $\{F_{1,n}\}=\{0,1,1,2,3,5,8,\ldots\}$ (A000045),
    \item $\{F_{2,n}\}=\{0,1,2,5,12,29,\ldots\}$ (A000129),
    \item $\{F_{3,n}\}=\{0,1,3,10,33,109,\ldots\}$ (A006190).
\end{itemize}

\noindent
As particular cases, when $k=1$ one recovers the classical Fibonacci sequence, whereas for $k=2$ the Pell sequence arises.

\par\smallskip
\noindent
Similarly, for any integer $k\geq 1$, the $k$-Lucas sequence $\{L_{k,n}\}_{n\geq 0}$ is defined by the recurrence
\[
L_{k,n+1}=kL_{k,n}+L_{k,n-1},\qquad n\geq 1,
\]
with initial conditions $L_{k,0}=2$ and $L_{k,1}=k$.

\par\smallskip
\noindent
The first values of the $k$-Lucas numbers are listed in Table~\ref{tab2}.

\begin{table}[ht]
\centering
\[
\begin{aligned}
& L_{k,0}=2, \\
& L_{k,1}=k, \\
& L_{k,2}=k^2+2, \\
& L_{k,3}=k^3+3k, \\
& L_{k,4}=k^4+4k^2+2, \\
& L_{k,5}=k^5+5k^3+5k, \\
& L_{k,6}=k^6+6k^4+9k^2+2, \\
& L_{k,7}=k^7+7k^5+14k^3+7k.
\end{aligned}
\]
\caption{First values of the $k$-Lucas numbers}
\label{tab2}
\end{table}

\noindent
Notable special cases include the classical Lucas sequence for $k=1$ and the Pell--Lucas sequence for $k=2$.

\par\smallskip
\noindent
Several $k$-Lucas sequences are also recorded in the OEIS \cite{OEIS}, including
\begin{itemize}
    \item $\{L_{1,n}\}=\{2,1,3,4,7,11,18,29,\ldots\}$ (A000032),
    \item $\{L_{2,n}\}=\{2,2,6,14,34,82,198,478,\ldots\}$ (A002203),
    \item $\{L_{3,n}\}=\{2,3,11,36,119,393,1298,4287,\ldots\}$ (A006497).
\end{itemize}

\par\smallskip
\noindent
The Fibonacci numbers, Lucas numbers, $k$-Fibonacci numbers, and $k$-Lucas numbers satisfy a wide variety of algebraic and combinatorial identities (see, for example, \cite{koshy,Falcon2011,FalconPlaza2009}). For the purposes of this paper, we focus on a small collection of identities that play a fundamental role in the analysis developed in the subsequent sections.
\par\smallskip

\noindent In particular, we establish and prove four key identities: two classical identities involving Fibonacci and Lucas numbers, and two corresponding identities for their generalizations in terms of $k$-Fibonacci and $k$-Lucas numbers. These identities play a fundamental role and will be repeatedly used in the derivation of closed-form expressions for the matrix properties studied throughout the paper.

\begin{thm}\label{restagen}
For $k \geq 1$, the following identity holds:
\[
F_{k,2n+1}L_{k,2n+2}-F_{k,2n+2}L_{k,2n+1}=2.
\]
\end{thm}

\begin{proof}
According to Theorem~2.4 in \cite{Falcon2011}, the $k$-Lucas numbers satisfy
\[
L_{k,n}=F_{k,n-1}+F_{k,n+1}, \qquad n \geq 1.
\]
Applying this identity for $n=2n+1$ and $n=2n+2$, we obtain
\[
L_{k,2n+2}=F_{k,2n+1}+F_{k,2n+3}
\]
and
\[
L_{k,2n+1}=F_{k,2n}+F_{k,2n+2}.
\]
Substituting these expressions into the left-hand side of the identity, we get
\begin{align*}
F_{k,2n+1}L_{k,2n+2}-F_{k,2n+2}L_{k,2n+1}
&=F_{k,2n+1}\left(F_{k,2n+1}+F_{k,2n+3}\right)
 -F_{k,2n+2}\left(F_{k,2n}+F_{k,2n+2}\right)\\
&=F_{k,2n+1}^2+F_{k,2n+1}F_{k,2n+3}
 -F_{k,2n+2}F_{k,2n}-F_{k,2n+2}^2.
\end{align*}
Rearranging terms, this expression can be written as
\[
F_{k,2n+1}F_{k,2n+3}-F_{k,2n+2}^2
-\left(F_{k,2n+2}F_{k,2n}-F_{k,2n+1}^2\right).
\]
By the Simson identity,
\[
F_{k,n-1}F_{k,n+1}-F_{k,n}^2=(-1)^n,
\]
and applying it for $n=2n+2$ and $n=2n+1$, respectively, we obtain
\[
F_{k,2n+1}L_{k,2n+2}-F_{k,2n+2}L_{k,2n+1}
=(-1)^{2n+2}-(-1)^{2n+1}=1-(-1)=2,
\]
which completes the proof.
\end{proof}

\begin{cor}\label{aux1}
For $k=1$, the following identity holds:
\[
F_{2n+1}L_{2n+2}-F_{2n+2}L_{2n+1}=2.
\]
\end{cor}

\begin{thm}\label{aux2}
The following identity holds:
\[
\sum_{i=1}^{n-1}F_{2i}L_{2i-1}+F_{2n-1}L_{2n}
=
2+\sum_{i=1}^{n}F_{2i}L_{2i-1}.
\]
\end{thm}

\begin{proof}
Starting from the right-hand side, we have
\[
2+\sum_{i=1}^{n}F_{2i}L_{2i-1}
=
2+F_{2n}L_{2n-1}+\sum_{i=1}^{n-1}F_{2i}L_{2i-1}.
\]
By Corollary~\ref{aux1}, we know that
\[
F_{2n}L_{2n-1}=F_{2n-1}L_{2n}-2.
\]
Substituting this expression into the previous equality yields
\[
2+\sum_{i=1}^{n}F_{2i}L_{2i-1}
=
2+F_{2n-1}L_{2n}-2+\sum_{i=1}^{n-1}F_{2i}L_{2i-1},
\]
which proves the desired identity.
\end{proof}

\begin{thm}\label{sumgen}
For $k \geq 1$, the following identity holds:
\[
\sum_{j=1}^{n}F_{k,2j}L_{k,2j-1}
=
\frac{F_{k,4n+3}-F_{k,4n-1}+F_{k,5}-1}{L_{k,4}-2}-1-n.
\]
\end{thm}

\begin{proof}
We recall the identity
\[
F_{k,m}L_{k,n}=F_{k,m+n}-(-1)^mF_{k,n-m}.
\]
Applying it with $m=2j$ and $n=2j-1$, we obtain
\[
F_{k,2j}L_{k,2j-1}
=
F_{k,4j-1}-(-1)^{2j}F_{k,-1}
=
F_{k,4j-1}-F_{k,-1}.
\]
Since $F_{k,-n}=(-1)^{n+1}F_{k,n}$, it follows that $F_{k,-1}=1$, and therefore
\[
F_{k,2j}L_{k,2j-1}=F_{k,4j-1}-1.
\]
Consequently,
\[
\sum_{j=1}^{n}F_{k,2j}L_{k,2j-1}
=
\sum_{j=1}^{n}F_{k,4j-1}-n.
\]
By Theorem~4 in \cite{FalconPlaza2009}, for $m=4$ and $p=-1$, we have
\[
\sum_{i=0}^{n}F_{k,4i-1}
=
\frac{F_{k,4n+3}-F_{k,4n-1}+F_{k,5}-1}{L_{k,4}-2}.
\]
Hence,
\[
\sum_{i=1}^{n}F_{k,4i-1}
=
\frac{F_{k,4n+3}-F_{k,4n-1}+F_{k,5}-1}{L_{k,4}-2}-1,
\]
and the result follows.
\end{proof}

\noindent Another key concept used in this paper is the spectral radius and the energy of a graph.
Let $A$ be an $n \times n$ matrix with real or complex entries and eigenvalues
$\lambda_1,\ldots,\lambda_n$.
The spectral radius $\rho(A)$ of $A$ is defined as
\[
\rho(A)=\max_{1 \leq i \leq n}|\lambda_i|,
\]
that is, the largest absolute value of its eigenvalues.
The spectral radius of a finite graph is defined as the spectral radius of its adjacency matrix,
and the energy of a graph is defined as the sum of the absolute values of the eigenvalues of its adjacency matrix.
This concept is related to the energy of certain classes of molecules in chemistry and was first introduced in mathematics by Gutman in 1978 \cite{Gutman}.
The present work contributes to the study of the spectral theory of matrices associated with integer sequences.

\section{Generic Case: Fibonacci and Lucas Numbers}\label{Generic}

In this section, we study, using combinatorial and elementary algebraic tools, a family of square matrices introduced by Michel Bataille (Rouen, France) in Elementary Problem~B1360 of The Fibonacci Quarterly (see \cite{fibonacci2024}). According to this problem, let $A_n = \left(a_{i,j}\right)$ be the $n \times n$ matrix whose entries are given by
\[
a_{i,i} = F_{2i-1} L_{2i},
\]
and
\[
a_{i,j} = F_{2j} L_{2j-1}, \quad i,j = 1,2,\ldots,n, \quad i \neq j.
\]
Here, $F_n$ and $L_n$ denote the Fibonacci and Lucas numbers, respectively. We provide explicit expressions for several classical properties of this family of matrices, including the determinant, inverse, trace, and matrix powers. Furthermore, we determine the spectral radius and the energy of the graphs associated with these matrices.

\subsection{Determinant}
\noindent Concerning the determinant, Michel Bataille established in Problem~B1360 that the following result holds.
\begin{thm}
Let $A_n=\left(a_{i, j}\right)$ be the $n \times n$ matrix defined above. Then its determinant is given by
\[
\operatorname{det}\left(A_n\right)=\frac{2^{n-1}\left(L_{4 n+1}+9-5 n\right)}{5}.
\]
\end{thm}

\begin{proof}
First, we simplify the matrix $A_n$ using elementary row operations. We add $-1$ times the first row $r_1$ of $A_n$ to each of the rows $r_2, r_3, \ldots, r_n$.

As established by the first result in \cite{Ferns}, the following identity holds:
\[
F_{x_1}L_{x_2} = F_{x_1 + x_2} + (-1)^{x_2} F_{x_1 - x_2}.
\]
Applying this formula with $x_1 = 2n$ and $x_2 = 2n - 1$, we obtain:
\[
F_{2n}L_{2n-1} = F_{4n - 1} - F_1.
\]
Similarly, taking $x_1 = 2n - 1$ and $x_2 = 2n$, we have:
\[
F_{2n - 1}L_{2n} = F_{4n - 1} + F_1.
\]
Subtracting these two expressions yields:
\[
F_{2n - 1}L_{2n} - F_{2n}L_{2n - 1} = 2F_1 = 2,
\]
and thus
\[
F_{2n}L_{2n - 1} - F_{2n - 1}L_{2n} = -2.
\]

Therefore, the reduced matrix $\bar{A}_n$ is an $n \times n$ matrix with entries defined by
\[
\bar{A}_n = (a_{ij}) =
\begin{cases}
F_1 L_2 & \text{if } i = j = 1, \\
F_{2j} L_{2j - 1} & \text{if } i = 1 \text{ and } 2 \leq j \leq n, \\
-2 & \text{if } 2 \leq i \leq n \text{ and } j = 1, \\
2 & \text{if } i = j \text{ and } 2 \leq i \leq n, \\
0 & \text{otherwise}.
\end{cases}
\]

According to Theorem~4.6 in \cite{Friedberg}, we have $\det(A_n) = \det(\bar{A}_n)$. Therefore, it suffices to prove that
\[
\det(\bar{A}_n) = \frac{2^{n - 1}(L_{4n + 1} + 9 - 5n)}{5}.
\]

We proceed by mathematical induction on $n$.
\par\bigskip

For $n = 2$, we compute
\[
\det(\bar{A}_2) = \det \begin{pmatrix}
F_1 L_2 & F_4 L_3 \\
-2 & 2
\end{pmatrix}
= \det \begin{pmatrix}
3 & 12 \\
-2 & 2
\end{pmatrix} = 30.
\]
On the other hand,
\[
\frac{2^{2 - 1}(L_9 + 9 - 10)}{5} = \frac{2(76 - 1)}{5} = \frac{2 \cdot 75}{5} = 30.
\]
Thus, the base case holds.

Now, assume that the identity holds for some fixed $n \in \mathbb{N}$, that is,
\[
\det(\bar{A}_n) = \frac{2^{n - 1}(L_{4n + 1} + 9 - 5n)}{5}.
\]
We aim to show that
\[
\det(\bar{A}_{n+1}) = \frac{2^n(L_{4n + 5} + 4 - 5n)}{5}.
\]

Since $\det(\bar{A}_{n+1})$ can be computed by cofactor expansion along any row, we expand along the last row:
\begin{align*}
\det(\bar{A}_{n+1}) &= 2^n(F_{2n+2}L_{2n+1}) + 2\det(\bar{A}_n) \\
&= 2^n\left(F_{2n+2} L_{2n+1} + \frac{L_{4n+1} + 9 - 5n}{5}\right) \\
&= \frac{2^n \left(5 F_{2n+2} L_{2n+1} + L_{4n+1} + 9 - 5n\right)}{5}.
\end{align*}

Using the identity $5F_k = L_{k - 1} + L_{k + 1}$ for $k = 2n+2$ (see \cite{koshy}, p.~90, Corollary~5.5), we obtain
\[
5F_{2n+2} = L_{2n+1} + L_{2n+3}.
\]
Thus,
\[
5 F_{2n+2} L_{2n+1} = L_{2n+1}^2 + L_{2n+1} L_{2n+3}.
\]
Substituting, we have
\[
\det(\bar{A}_{n+1}) = \frac{2^n}{5} \left(L_{2n+1}^2 + L_{2n+1} L_{2n+3} + L_{4n+1} + 9 - 5n\right).
\]

According to \cite{koshy}, p.~97, Exercise~41, the identity $L_k^2 = L_{2k} + 2(-1)^k$ holds. In particular, for $k = 2n + 1$,
\[
L_{2n+1}^2 = L_{4n+2} - 2.
\]
Substituting again, we obtain
\[
\det(\bar{A}_{n+1}) = \frac{2^n}{5} \left(L_{4n+2} - 2 + L_{2n+1} L_{2n+3} + L_{4n+1} + 9 - 5n\right).
\]

By the recursive definition of the Lucas numbers, this simplifies to
\[
\det(\bar{A}_{n+1}) = \frac{2^n}{5} \left(L_{4n+3} + L_{2n+1} L_{2n+3} + 7 - 5n\right).
\]

Next, using the identity $L_{n-1}L_{n+1} = 5(-1)^{n-1} + L_n^2$ for $n = 2n + 2$ (see \cite{koshy}, p.~97, Exercise~38), we find that
\[
L_{2n+1} L_{2n+3} = -5 + L_{2n+2}^2,
\]
and hence
\[
\det(\bar{A}_{n+1}) = \frac{2^n}{5} \left(L_{4n+3} + L_{2n+2}^2 + 2 - 5n\right).
\]

Using again the identity $L_k^2 = L_{2k} + 2(-1)^k$ with $k = 2n+2$, we obtain
\[
\det(\bar{A}_{n+1}) = \frac{2^n}{5} \left(L_{4n+3} + L_{4n+4} + 4 - 5n\right).
\]

Finally, applying the Lucas recurrence relation
\[
L_{4n+3} + L_{4n+4} = L_{4n+5},
\]
we conclude that
\[
\det(\bar{A}_{n+1}) = \frac{2^n (L_{4n+5} + 4 - 5n)}{5}.
\]

This completes the inductive step and, therefore, the proof.
\end{proof}

\begin{rem}
The integer sequence $3, 30, 412, 5696, 78272, \ldots$ arising from the determinant of $A_n$ does not appear in the On-Line Encyclopedia of Integer Sequences (OEIS).
\end{rem}

\subsection{Trace}
For the trace, the following result holds.
\begin{thm}\label{tr}
The trace of the matrix $A_n$, denoted by $\mathrm{Tr}(A_n)$, is given by
\[
\mathrm{Tr}(A_n)=\frac{F_{4n+3}-F_{4n-1}+4}{5}+n-1.
\]
\end{thm}

\begin{proof}
According to the definition of the trace, it is given by $\displaystyle\sum_{i=1}^{n} F_{2i-1} L_{2i}$. 
Applying the identity $F_m L_n = F_{m+n} + (-1)^n F_{m-n}$ with $m = 2i - 1$ and $n = 2i$, we obtain
$F_{2i-1} L_{2i} = F_{4i-1} + F_{-1}$ for each summand in the trace.
Recalling that $F_{-1} = F_1$, the expression can be simplified as:

\begin{align*}\label{tr}
    \displaystyle\sum_{i=1}^{n}F_{2i-1}L_{2i}=\displaystyle\sum_{i=1}^{n}F_{4n-1}+n\cdot F_1
\end{align*}

\noindent Now, using Theorem~5.11 of \cite{koshy} for $k = 4$ and $j = -1$, we have
\[
\sum_{i=1}^{n} F_{4i-1}
= \frac{F_{4n+3} - F_{4n-1} + 4}{5} - 1.
\]
Thus, the trace is given by
\begin{align*}
\sum_{i=1}^{n} F_{2i-1} L_{2i}
&= \frac{F_{4n+3} - F_{4n-1} + 4}{5} + n - 1 \\
&= \frac{F_{4n+3} - F_{4n-1} + 5n - 1}{5}.
\end{align*}
\end{proof}

\begin{rem}
The integer sequence $3,17,107,718,4900,33558,\ldots$ arising from the trace of $A_n$ is not currently listed in the On-Line Encyclopedia of Integer Sequences (OEIS).
\end{rem}

\subsection{Eigenvalues, Spectrum and some Consequences}
Regarding the characteristic polynomial, we establish the following result:

\begin{lem}\label{charp}
The characteristic polynomial $p(\lambda)$ associated with $A_n$ is given by
\[
p(\lambda)=-(2-\lambda)^{n-1}\left(\lambda-\displaystyle\sum_{i=1}^{n-1}F_{2i}L_{2i-1}+F_{2n-1}L_{2n}\right).
\]
\end{lem}

\begin{proof} We proceed by induction: \\
For $n=2$ we have: 
\begin{align*}
\left|A_n-\lambda I_n\right|&=\left|\left(\begin{array}{ll}
3 & 12 \\
1 & 14
\end{array}\right)-\left(\begin{array}{ll}
\lambda & 0 \\
0 & \lambda
\end{array}\right)\right| \\
&=\left|\left(\begin{array}{cc}
3-\lambda & 12 \\
1 & 14-\lambda
\end{array}\right)\right|\\
& =(3-\lambda)(14-\lambda)-12 =42-3 \lambda-14 \lambda+\lambda^2-12 \\
& =\lambda^2-17 \lambda+30  =(\lambda-2)(\lambda-15) =-(2-\lambda)(\lambda-15)
\end{align*}
On the other hand note that $F_{2}L_{1}+F_{3}L_{4}=1+14=15$. Thus, the base case holds.    

\noindent Now, we assume that the identity holds for some fixed $n\in \mathbb{N}$, we mean that, 

$$
\operatorname{det}\left(A_n-\lambda I_n\right) =-(2-\lambda)^{n-1}\left(\lambda-\sum_{i=1}^{n-1}F_{2 i}L_{2i-1}+F_{2n-1}L_{2 n}\right)
$$
\noindent We want to show that: 
$$
\operatorname{det}\left(A_{n+1}-\lambda I_{n+1}\right)=-(2-\lambda)^n\left(\lambda-\sum_{i=1}^nF_{2i} L_{2 i-1}+F_{2 n+1} L_{2n+2}\right)
$$

\noindent We start by computing $\operatorname{det}\left(B\right)$, where $B=A_{n+1}-\lambda I_{n+1}$, so, we want to compute  

$$\left|\begin{array}{ccccccc}
F_1 L_2-\lambda & F_4 L_3 & F_6 L_5 & F_8 L_7 & \cdots & F_{2 n} L_{2 n-1} & F_{2 n+2} L_{2 n+1} \\
F_2 L_1 & F_3 L_4-\lambda & F_6 L_5 & F_8 L_7 & \cdots & F_{2 n} L_{2 n-1} & F_{2 n+2} L_{2 n+1} \\
F_2 L_1 & F_4 L_3 & F_5L_6-\lambda & F_8 L_7 & \cdots & F_{2 n} L_{2 n-1} & F_{2 n+2} L_{2 n+1} \\
F_2 L_1 & F_4 L_3 & F_6 L_5 & F_7 L_8-\lambda & \cdots & F_{2 n} L_{2 n-1} & F_{2 n+2} L_{2 n+1} \\
\vdots & \vdots & \vdots & \vdots & \ddots & \vdots & \vdots \\
F_2 L_1 & F_4 L_3 & F_6 L_5 & F_8 L_7 & \cdots & F_{2 n-1} L_n-\lambda & F_{2 n+2} L_{2 n+1} \\
F_2 L_1 & F_4 L_3 & F_6 L_5 & F_8 L_7 & \cdots & F_{2 n} L_{2 n-1} & F_{2 n+1} L_{2 n+2}-\lambda
\end{array}\right|$$

\noindent Since $\operatorname{det}\left(A_{n+1}-\lambda I_{n+1}\right)$ can be calculated by expanding by cofactors using any row of $B$ we compute $\operatorname{det}\left(A_{n+1}-\lambda I_{n+1}\right)$ with the last row: 
\begin{align*}
\operatorname{det}(B)
&= \operatorname{det}(A_{n+1}-\lambda I_{n+1}) \\
&= (-1)^{n+2}F_2L_1|M_1|
 + (-1)^{n+3}F_{4}L_{3}|M_2| + \cdots \\
&\quad + (-1)^{2n+1}F_{2n}L_{2n-1}|M_n| \\
&\quad + (-1)^{2n+2}\bigl(F_{2n+1}L_{2n+2}-\lambda\bigr)|M_{n+1}|.
\end{align*}

\noindent where $M_k$ is an $n \times n$ matrix obtained from the matrix $B$ by deleting the $k$-th column and the $(n+1)$-st row, and which can be explicitly described as
$$
M_k=(m_{i,j})=\left\{\begin{array}{lll}
b_{i,j} & \text { if } & j<k \\
b_{i,j+1} & \text {if} & k \leq j \leq n
\end{array}\right.
$$

\noindent when $1\leq k\leq n$ where $b_{i,j}$ denoted the corresponding entry of the matrix $B$ and $1\leq i \leq n$ and $M_{n+1}$ is equal to $A_{n}-\lambda I_n$.

\noindent The above shows that, in order to compute $\operatorname{det}\left(A_{n+1}-\lambda I_{n+1}\right)$, it suffices to determine the value of $M_k$ for $1 \leq k \leq n$. To this end, we consider the matrix $\overline{M}_k$, which is obtained from $M_k$ by applying the admissible row transformations consisting of adding $-1$ times the row $r_k$ to every row $r_t$, for all $t \neq k$. These admissible transformations allow us to describe $\overline{M}_k$ explicitly as follows:

$$
\overline{M_k}=(b_{i,j})= \begin{cases}b_{ij} & \text { si } i=k \\ 
2-\lambda & \text { si } i=j  \text{ and } i<k \\ 
2-\lambda & \text { si } j=i-1 \text { and } i>k \\ 
0 & \text { otherwise }\end{cases}
$$
Thus,
\begin{align*}
\operatorname{det}(B)
&= \operatorname{det}(A_{n+1}-\lambda I_{n+1}) \\
&= (-1)^{n+2}F_2L_1|\overline{M}_1|
 + (-1)^{n+3}F_{4}L_{3}|\overline{M}_2| + \cdots \\
&\quad + (-1)^{2n+1}F_{2n}L_{2n-1}|\overline{M}_n| \\
&\quad + (-1)^{2n+2}\bigl(F_{2n+1}L_{2n+2}-\lambda\bigr)|A_n-\lambda I_n|.
\end{align*}

\noindent To compute $|\overline{M}_k|$ we again expand by cofactors using the row $k$ of the corresponding $\overline{M}_k$, that is, $|\overline{M}_k|=\displaystyle \sum_{j=1}^{n}b_{kj}|P_j|$ where $P_j$ is a matrix of size $(n-1)\times(n-1)$ with the following form
\[
P_j = (P_{s,t})=
\begin{cases}
2 - \lambda, & \text{if } s = t \text{ and } s < j, \\
2 - \lambda, & \text{if } t = s - 1 \text{ and } s > j, \\
0,           & \text{otherwise}.
\end{cases}
\]

\noindent Note that the column $n-1$ is always $0$ for all $P_j$ with $1\leq j \leq n-1$. By properties of determinants $|P_j|=0$ for $1\leq j \leq n-1$. Thus, $|\overline{M}_k|=(-1)^{n+k}b_{k,n}|P_n|$, but $P_n=I_{n-1}(2-\lambda)$ which implies that: 
\[|\overline{M_{k}}|=(-1)^{n+k}F_{2n+2}L_{2n+1}(2-\lambda)^{n-1}\]
So, applying the induction hypothesis we get: 

\begin{align}
\operatorname{det}(B)
&= \operatorname{det}(A_{n+1}-\lambda I_{n+1}) \notag \\
&= (-1)^{n+2}F_{2}L_{1}\cdot(-1)^{n+1}F_{2n+2}L_{2n+1}(2-\lambda)^{n-1} \notag \\
&\quad + (-1)^{n+3}F_{4}L_{3}\cdot(-1)^{n+2}F_{2n+2}L_{2n+1}(2-\lambda)^{n-1} \notag \\
&\quad + \cdots \notag \\
&\quad + (-1)^{2n+1}F_{2n}L_{2n-1}\cdot(-1)^{2n}F_{2n+2}L_{2n+1}(2-\lambda)^{n-1} \notag \\
&\quad + (-1)^{2n+2}\bigl(F_{2n+1}L_{2n+2}-\lambda\bigr) \cdot -(2-\lambda)^{n-1}
\left(\lambda-\sum_{i=1}^{n-1}F_{2i}L_{2i-1}+F_{2n-1}L_{2n}\right)\notag
\end{align}

\noindent which can be simplify and factorize by 
\[
-(2-\lambda)^{n-1}\Bigl(
F_{2n+2}L_{2n+1}\Bigl( \sum_{i=1}^{n} F_{2i}L_{2i-1} \Bigr)
+\bigl(F_{2n+1}L_{2n+2}-\lambda\bigr)
\bigl(\lambda-\sum_{i=1}^{n-1}F_{2i}L_{2i-1}+F_{2n-1}L_{2n}\bigr)
\Bigr).
\]

\noindent Simplifying and reordering 

\[
\begin{aligned}
-(2-\lambda)^{n-1}\Bigl(
&-\lambda^{2}
+ \lambda\Bigl(
F_{2n+1}L_{2n+2}
+ \sum_{i=1}^{n-1} F_{2i}L_{2i-1}
+ F_{2n-1}L_{2n}
\Bigr) \\
&- F_{2n+1}L_{2n+2}
\Bigl(
\sum_{i=1}^{n-1} F_{2i}L_{2i-1}
+ F_{2n-1}L_{2n}
\Bigr) \\
&+ F_{2n+2}L_{2n+1}
\Bigl(
\sum_{i=1}^{n} F_{2i}L_{2i-1}
\Bigr)
\Bigr).
\end{aligned}
\]

\noindent Now, we use Corollary \ref{aux1} and Theorem \ref{aux2} to simplify the last expression: 

\begin{align*}
\operatorname{det}(B)&=-(2-\lambda)^{n-1} \Bigl( -\lambda^2+\lambda\left(F_{2 n+1} L_{2 n+2}+\sum_{i=1}^{n-1} F_{2 i} L_{2 i-1}+F_{2 n-1} L_{2 n}\right)\\
&-F_{2 n+1} L_{2 n+2} \cdot\left(\sum_{i=1}^{n-1} F_{2 i} L_{2 i-1}+F_{2 n-1} \cdot L_{2 n}\right)\\ 
&+F_{2 n+2} L_{2 n+1}\left(\sum_{i=1}^n F_{2 i} L_{2 i-1}\right)\Bigr)\\
&= -(2-\lambda)^{n-1} \Bigl(-\lambda^2+\lambda\left(F_{2 n+1} L_{2 n+2}+\sum_{i=1}^{n-1} F_{2i} L_{2 i-1}+F_{2 n-1} L_{2 n}\right)\\
&-\Bigl(\left(2+F_{2 n+2} L_{2 n+1}\right) \cdot\left(2+\sum_{i=1}^n F_{2 i} L_{2 i-1}\right)\Bigr) \\
&  +F_{2 n+2} L_{2 n+1} \cdot\left(\sum_{i=1}^n F_{2 i} L_{2 i-1}\right)\Bigr)
\end{align*}
\normalcolor
The second factor is equivalent to 

\begin{align*}
&=-\lambda^2+\lambda(F_{2n+1}L_{2n+2}+\sum_{i=1}^{n-1}F_{2i}L_{2i-1}+F_{2n-1}L_{2n})\\
&-\Bigl(2\cdot2+2\sum_{i=1}^{n}F_{2i}L_{2i-1}+2F_{2n+2}L_{2n+1}+F_{2n+2}L_{2n+1}\cdot\sum_{i=1}^{n}F_{2i}L_{2i-1}\Bigr)\\
&+F_{2n+2}L_{2n+1}\Bigl(\sum_{i=1}^{n}F_{2i}L_{2i-1}\Bigr)
\end{align*}
\begin{align*}
&=-\lambda^2+\lambda(F_{2n+1}L_{2n+2}+\sum_{i=1}^{n-1}F_{2i}L_{2i-1}+F_{2n-1}L_{2n})\\
&-2\cdot2-2\sum_{i=1}^{n}F_{2i}L_{2i-1}-2F_{2n+2}L_{2n+1}-F_{2n+2}L_{2n+1}\cdot\sum_{i=1}^{n}F_{2i}L_{2i-1}\\
&+F_{2n+2}L_{2n+1}\Bigl(\sum_{i=1}^{n}F_{2i}L_{2i-1}\Bigr)
\end{align*}
\begin{align*}
&=-\lambda^2+\lambda(F_{2n+1}L_{2n+2}+\sum_{i=1}^{n-1}F_{2i}L_{2i-1}+F_{2n-1}L_{2n})\\
&-2\Bigl( 2+\sum_{i=1}^{n}F_{2i}L_{2i-1}+F_{2n+2}L_{2n+1}\Bigr)
\end{align*}
\begin{align*}
&=-\lambda^2+\lambda(F_{2n+1}L_{2n+2}+\sum_{i=1}^{n-1}F_{2i}L_{2i-1}+F_{2n-1}L_{2n})\\
&-2\Bigl( \sum_{i=1}^{n}F_{2i}L_{2i-1}+F_{2n+1}L_{2n+2}\Bigr)
\end{align*}
\noindent Note that by Theorem \ref{aux2} 

\[\sum_{i=1}^{n-1}F_{2i}L_{2i-1}+F_{2n-1}L_{2n}+F_{2n+1}L_{2n+2}=2+\sum_{i=1}^{n}F_{2i}L_{2i-1}+F_{2n+1}L_{2n+2}\]

\noindent So the expression becomes

\[-\lambda^2+\lambda(2+\sum_{i=1}^{n}F_{2i}L_{2i-1}+F_{2n+1}L_{2n+2})\\
-2\Bigl( \sum_{i=1}^{n}F_{2i}L_{2i-1}+F_{2n+1}L_{2n+2}\Bigr)\]
\noindent It can be factor as

\[(2-\lambda)\Bigl( \lambda - \sum_{i=1}^{n} F_{2i}L_{2i-1}+F_{2n+1}L_{2n+2} \Bigr)\]

\noindent Thus 
\[\operatorname{det}(A_{n+1}-\lambda I_{n+1})=-(2-\lambda)^{n-1}\cdot(2-\lambda)\cdot\Bigl( \lambda - \sum_{i=1}^{n} F_{2i}L_{2i-1}+F_{2n+1}L_{2n+2} \Bigr)\]
So, 
\[\operatorname{det}(A_{n+1}-\lambda I_{n+1})=-(2-\lambda)^{n}\cdot\Bigl( \lambda - \sum_{i=1}^{n} F_{2i}L_{2i-1}+F_{2n+1}L_{2n+2} \Bigr)\]

\noindent This completes the inductive step, and thus the proof.
\end{proof}

\begin{thm}\label{eigen}
The matrix $A_n$ has exactly two distinct eigenvalues:
\[
\frac{F_{4n + 3} - F_{4n - 1} - 5n + 9}{5} \quad \text{(with multiplicity 1)} \quad \text{and} \quad
2 \quad \text{(with multiplicity } n - 1\text{)}.
\]
\end{thm}

\begin{proof}
According to Theorem \ref{charp}, the characteristic polynomial is given by
\[
p(\lambda)=-(2-\lambda)^{n-1}\left(\lambda-\displaystyle\sum_{i=1}^{n-1}F_{2i}L_{2i-1}+F_{2n-1}L_{2n}\right).
\]
If we equate the characteristic polynomial to zero, we obtain
\[
p(\lambda)=-(2-\lambda)^{n-1}\left(\lambda-\displaystyle\sum_{i=1}^{n-1}F_{2i}L_{2i-1}+F_{2n-1}L_{2n}\right)=0.
\]

\noindent This implies that $A_n$ has two eigenvalues: $\lambda_{1}=2$ with multiplicity $n-1$, and
\[
\lambda_{2}=\displaystyle\sum_{i=1}^{n} F_{2i}L_{2i-1}+F_{2n-1}L_{2n}.
\]
We note that $\lambda_2$ can be reduced by using identities involving Fibonacci numbers. As in Theorem \ref{tr}, we use the identity $F_{m}L_{n}=F_{m+n}+(-1)^nF_{m-n}$ with $m=2i$ and $n=2i-1$ for each term in the sum. This implies that $F_{2n-2}L_{2n-3}=F_{4n-5}-F_{1}$, and therefore the sum can be written as
\[
\displaystyle\sum_{i=1}^{n} F_{2i}L_{2i-1} =\sum_{i=1}^{n-1}F_{4i-1}-(n-1)F_{1}.
\]
Similarly, for the second term in $\lambda_{2}$, using $F_{m}L_{n}=F_{m+n}+(-1)^nF_{m-n}$ with $m=2n-1$ and $n=2n$, we obtain
\[
F_{2n-1}L_{2n}=F_{4n-1}+F_{1}.
\]
Thus, the expression can be written as
\begin{align*}
    \lambda_2&=\displaystyle\sum_{i=1}^{n} F_{2i}L_{2i-1}+F_{2n-1}L_{2n}\\
    &= \displaystyle \sum_{i=1}^{n-1} F_{4i-1}-(n-1)F_1+F_{4n-1}+F_1\\
    &=\displaystyle \sum_{i=1}^{n} F_{4i-1}-(n-2).
\end{align*}

\noindent By Theorem 5.11 of \cite{koshy}, we have
\begin{align*}
    \lambda_2&=\displaystyle \sum_{i=1}^{n} F_{4i-1}-(n-2)\\
    &=\frac{F_{4n+3}-F_{4n-1}+4}{5}-1-(n-2)\\
    &=\frac{F_{4n+3}-F_{4n-1}+4}{5}-n+1\\
    &=\frac{F_{4n+3}-F_{4n-1}-5n+9}{5}.
\end{align*}
\end{proof}

\begin{rem}
The integer sequence $3,15,103,712,4892,33548\ldots$ arising from the eigenvalues of $A_n$ is not encoded in the On-Line Encyclopedia of Integer Sequences (OEIS).
\end{rem}

\noindent Regarding the spectral radius and the energy, we obtain the following results:

\begin{cor}
The spectral radius $\rho(A_n)$ of the matrix $A_n$, with $n\geq2$, is given by
\[
\frac{F_{4n + 3} - F_{4n - 1} - 5n + 9}{5}.
\]
\end{cor}

\begin{proof}
It is a direct consequence of Theorem \ref{eigen}.
\end{proof}

\noindent Let us consider the multigraph $G_{A_n}$ which has $A_n$ as its adjacency matrix. By the definition of the energy of a directed graph, we obtain the following result.

\begin{cor}
The energy of $G_{A_{n}}$ is given by
\[
\frac{F_{4n+3}-F_{4n-1}+5n-1}{5}.
\]
\end{cor}

\begin{proof}
It is a direct consequence of Theorem \ref{eigen}.
\end{proof}

\begin{rem}
The energy of $G_{A_{n}}$ coincides with the trace of $A_{n}$.
\end{rem}

\noindent In the next section, we study other properties of the matrices $A_n$; see Section \ref{powersgen} and Section \ref{inversegen}.

\section{General Results: $k$-Fibonacci and $k$-Lucas Numbers}\label{General}

In this section, we introduce the following family of matrices.

\begin{defin}
Let $A_{k,n}=\left(a_{i, j}\right)$ be the $n \times n$ matrix whose entries are given by $a_{i, i}=F_{k,2i-1} L_{k,2 i}$, and
\[
a_{i,j}=F_{k,2 j} L_{k,2j-1}, \quad i, j=1,2, \ldots, n, \quad i \neq j.
\]

\noindent Here, $F_{k,n}$ and $L_{k,n}$ denote the $k$-Fibonacci and $k$-Lucas numbers, respectively (see Section \ref{preliminaries}).
\end{defin}

\noindent As in Section \ref{Generic}, we provide explicit expressions for several classical properties of this family of matrices, including the determinant, inverse, trace, and matrix powers. Furthermore, we determine the spectral radius and the energy of the graphs associated with these matrices.

\subsection{Determinant}
For the determinant, the following result holds.

\begin{thm}\label{detgen}
Let $A_{k,n}$ be the $n \times n$ matrix defined as above. Then, its determinant is given by
\[
\det(A_{k,n})=2^n\left(1+\frac{1}{2}\left( \frac{F_{k,4n+3}-F_{k,4n-1}+F_{k,5}-1}{L_{k,4}-2}-1-n\right)\right).
\]
\end{thm}

\begin{proof}
First, according to Theorem \ref{restagen}, the identity
$F_{k,2j-1}L_{k,2j}-F_{k,2j}L_{k,2j-1}=2$ holds for all $j \geq 1$. 
Therefore, the matrix $A_{k,n}$ admits the following decomposition:
\[
A_{k,n}=D+\mathbf{1}\cdot v^{T}.
\]

\noindent Here, $D=\operatorname{diag}(2,2,\cdots,2)$, 
$v=\left(F_{k,2}L_{k,1},F_{k,4}L_{k,3},\cdots, F_{k,2n}L_{k,2n-1} \right)^{T}_{n\times 1}$, 
and $\mathbf{1}=(1,1,\cdots,1)_{n\times 1}^{T}$. 
Moreover, $\mathbf{1}v^{T}$ denotes the $n \times n$ matrix obtained as the product of $\mathbf{1}$ and $v^{T}$.

\noindent Note that $D$ is invertible. Thus, applying Cauchy’s determinant identity 
(see Example~1.3.24 in \cite{HornJohnson}, p.~66), we obtain:
\begin{align*}
    \det(D+\mathbf{1}v^{T})&= \det(D)\cdot(1+v^{T}D^{-1}\cdot\mathbf{1})\\
    &=\left(\prod_{i=1}^{n} 2 \right) \left(1 + v^{T}\cdot \frac{1}{2}I_{n}\cdot \mathbf{1}\right)\\
    &= 2^{n}\left(  1+ \frac{1}{2}\left( F_{k,2}L_{k,1}+F_{k,4}L_{k,3}+\cdots+ F_{k,2n}L_{k,2n-1} \right) \right).
\end{align*}

\noindent By Theorem \ref{sumgen}, we conclude that
\[
\det(A_{k,n})=\det(D+\mathbf{1}v^{T})=2^n\left(1+\frac{1}{2}\left( \frac{F_{k,4n+3}-F_{k,4n-1}+F_{k,5}-1}{L_{k,4}-2}-1-n\right)\right).
\]

\noindent This completes the proof.
\end{proof}

\noindent Theorem \ref{detgen} generates the following integer sequences, which are not encoded in the On-Line Encyclopedia of Integer Sequences (OEIS).

\begin{table}[h!]
\begin{tabular}{|l|l|}
\hline $k$ & Sequence $\det(A_{n, k})$ \\
\hline 2 & 6, 348, 23656, 1607504, 109216736, 7420311232, 504144305280, ...  \\
\hline 3 & 11, 2398, 570716, 135821824, 32323315136, 7692405726592, 1830663269698880, ... \\
\hline 4 & 18, 10980, 7071112, 4553754896, 2932589879072, 1888569667134016 ...    \\
\hline 5 & 27, 37854, 55039708, 80027590016, 116359895748032, 169186968307348864, ...  \\

\hline 6 & 38, 106780, 307953512, 888137513296, 2561387356578272, 7387037583821822656, ...  \\
\hline
\end{tabular}
\captionsetup{labelformat=empty}
\caption{Table 1. Particular cases of $\det(A_{k,n})$ for an arbitrary $k$ and $n\geq 1$.}
\end{table}

\normalcolor

\subsection{Trace}

For the trace, the following result holds.

\begin{thm}\label{tracegen}
The trace of the matrix $A_{k,n}$ is given by
\[
\mathrm{Tr}(A_{k,n})=\frac{F_{k,4n+3}-F_{k,4n-1}+F_{k,5}-1}{L_{k,4}-2}+n-1.
\]
\end{thm}

\begin{proof}
By definition, the trace of a square matrix is the sum of its diagonal entries. Hence,
\[
\operatorname{tr}(A_{k,n})
=
\sum_{i=1}^n a_{ii}
=
\sum_{i=1}^n F_{k,2i-1}L_{k,2i}.
\]

\noindent Using the identity $F_{k,m}L_{k,n}=F_{k,m+n}-(-1)^{m}F_{k,n-m}$ with $m=2j-1$ and $n=2j$, we obtain
\[
F_{k,2j-1}L_{k,2j}=F_{k,4j-1}+F_{k,1}
\]
for every summand of the trace. Since $F_{k,1}=1$, the expression can be simplified as
\[
\sum_{i=1}^n F_{k,2i-1}L_{k,2i}
=
\sum_{j=1}^{n}F_{k,4j-1}+n.
\]

\noindent Now, using Theorem~4 of \cite{FalconPlaza2009} for $a=4$ and $r=-1$, we have
\[
\sum_{j=1}^{n}F_{k,4j-1}
=
\frac{F_{k,4n+3}-F_{k,4n-1}-1+F_{k,5}}{L_{k,4}-2}-1.
\]

\noindent Thus, the trace is given by
\[
\mathrm{Tr}(A_{k,n})
=
\frac{F_{k,4n+3}-F_{k,4n-1}+F_{k,5}-1}{L_{k,4}-2}+n-1,
\]
as claimed.
\end{proof}

\noindent Theorem~\ref{tracegen} generates the following integer sequences which are not encoded in the On-Line Encyclopedia of Integer Sequences (OEIS).

\begin{table}[h!]
\begin{tabular}{|l|l|}
\hline $k$ & Sequence $\mathrm{Tr}(A_{n, k})$ \\
\hline 2 &  6, 176, 5918, 200944, 6826054, 231884736, 7877254782, ...   \\
\hline 3 & 11, 1201, 142683, 16977734, 2020207204, 240387678966, 28604113589057, ... \\
\hline 4 & 18, 5492, 1767782, 569219368, 183286867450, 59017802097948, ...    \\
\hline 5 & 27, 18929, 13759931, 10003448758, 7272493484260, 5287092759604662, ...  \\

\hline 6 & 38, 53392, 76988382, 111017189168, 160086709786150, ...  \\
\hline
\end{tabular}
\captionsetup{labelformat=empty}
\caption{Table 1. Particular cases of $\mathrm{Tr}(A_{n, k})$ for an arbitrary $k$ and $n\geq 1$.}
\end{table}

\subsection{Spectral Radius}
Regarding the spectrum of the matrix $A_{k,n}$, we have the following result.

\begin{thm}\label{eigegen}
The matrix $A_{k,n}$ has exactly two distinct eigenvalues:
\[
\lambda_{1}=2 
\quad \text{(with multiplicity } n - 1\text{)}, 
\quad \text{and} \quad 
\lambda_{2}=\frac{F_{k,4n + 3} - F_{k,4n - 1} + F_{k,5}-1}{L_{k,4}-2}-n+1
\quad \text{(with multiplicity 1)}.
\]
\end{thm}

\begin{proof}
First, we recall that according to the proof of Theorem \ref{detgen}, the matrix $A_{k,n}$ admits the decomposition
$A_{k,n}=D+\mathbf{1}\cdot v^{T}$, where
$D=\operatorname{diag}(2,2,\cdots,2)$,
$v=\left(F_{k,2}L_{k,1},F_{k,4}L_{k,3},\cdots, F_{k,2n}L_{k,2n-1} \right)_{n\times 1}$,
and $\mathbf{1}=(1,1,\cdots,1)_{n\times 1}^{T}$.
Moreover, $\mathbf{1}v^{T}$ is the $n\times n$ matrix obtained as the product of $\mathbf{1}$ and $v^{T}$.
By definition, a scalar $\lambda \in \mathbb{C}$ is an eigenvalue of a matrix $A$ if and only if
\[
\det(A - \lambda I) = 0.
\]

\noindent Substituting $A$ by $A_{k,n}$ and assuming that $\lambda \neq 2$, we obtain
\begin{align*}
    A_{k,n}-\lambda I_{n}&= D+\mathbf{1}v^{T}-\lambda I_{n}\\ 
    &=2I_{n}+\mathbf{1}\cdot v^{T}-\lambda I_{n}\\
    &=(2-\lambda)I_{n}+\mathbf{1}\cdot v^{T}\\
    &=(2-\lambda)I_{n} \cdot\left( I_{n}+\frac{1}{2-\lambda}\mathbf{1}\cdot v^{T}\right).
\end{align*}

\noindent Thus,
\begin{align*}
\det(A_{k,n}-\lambda I_{n})&=\det\left((2-\lambda)I_{n}\cdot \left(I_{n}+\frac{1}{2-\lambda}\mathbf{1}v^{T}\right)\right)\\
&=\det\left((2-\lambda)I_n\right)\cdot\det\left(I_{n}+\frac{1}{2-\lambda}\mathbf{1}v^{T}\right).
\end{align*}

\noindent Applying Cauchy’s determinant identity (see Example 1.3.24 in \cite{HornJohnson}, page 66), we obtain for the second term
\[
\det\left(I_{n}+\frac{1}{2-\lambda}\mathbf{1}v^{T}\right)
=
1+\frac{1}{2-\lambda}v^{T}\cdot\mathbf{1}.
\]
Thus,
\begin{align*}
\det(A_{k,n}-\lambda I_{n})&=(2-\lambda)^{n}\cdot\left(1+\frac{1}{2-\lambda}v^{T}\mathbf{1}\right)\\
&=(2-\lambda)^{n}+(2-\lambda)^{n-1}v^{T}\mathbf{1}\\
&=(2-\lambda)^{n-1}\bigl(2-\lambda+v^{T}\mathbf{1}\bigr).
\end{align*}

\noindent Since this identity holds for all $\lambda \in \mathbb{C}$, the characteristic equation
$\det(A_{k,n} - \lambda I) = 0$ is equivalent to
\[
(2-\lambda)^{n-1}\bigl(2-\lambda + v^{T}\mathbf{1}\bigr) = 0.
\]

\noindent Hence, the eigenvalues of $A_{k,n}$ are $\lambda_{1} = 2$ with algebraic multiplicity $n-1$
and $\lambda_{2} = 2 + v^{T}\mathbf{1}$ with algebraic multiplicity 1.
According to Theorem \ref{sumgen}, this completes the proof.
\end{proof}

\noindent Theorem \ref{eigegen} generates the following integer sequences, which are not encoded in the On-Line Encyclopedia of Integer Sequences (OEIS).

\begin{table}[h!]
\begin{tabular}{|l|l|}
\hline $k$ & Sequence of $\lambda_{2}$ \\
\hline 2 & 6, 174, 5914, 200938, 6826046, 231884726, 7877254770    \\
\hline 3 &  11, 1199, 142679, 16977728, 2020207196, 240387678956, 28604113589045 \\
\hline 4 &  18, 5490, 1767778, 569219362, 183286867442, 59017802097938,   \\
\hline 5 & 27, 18927, 13759927, 10003448752, 7272493484252, 5287092759604652,... \\
\hline 6 & 38, 53390, 76988378, 111017189162, 160086709786142, ...  \\
\hline
\end{tabular}
\captionsetup{labelformat=empty}
\caption{Table 1. Particular cases of $\lambda_{2}$ for an arbitrary $k$ and $n\geq 1$.}
\end{table}

\noindent Regarding the spectral radius and the energy, we have the following results.

\begin{cor}
The spectral radius $\rho(A_{k,n})$ of the matrix $A_{k,n}$, with $n \geq 2$, is given by
\[
\frac{F_{k,4n + 3} - F_{k,4n - 1} + F_{k,5} - 1}{L_{k,4}-2} - n + 1.
\]
\end{cor}

\begin{proof}
It is a direct consequence of Theorem \ref{eigegen}.
\end{proof}

\noindent Now, consider the multigraph $G_{A_n}$ having $A_n$ as its adjacency matrix.
From the definition of the energy of a directed graph, we obtain the following result.

\begin{cor}
The energy of $G_{A_{n}}$ is
\[
\frac{F_{k,4n + 3} - F_{k,4n - 1} + F_{k,5} - 1}{L_{k,4}-2} + n - 1.
\]
\end{cor}

\begin{proof}
It is a direct consequence of Theorem \ref{eigegen}.
\end{proof}

\subsection{Powers}\label{powersgen}

Throughout this section, we consider $P=\mathbf{1}v^{T}$, and we note that Lemma \ref{P} also holds for matrices of type $A_{k,n}$. For the power $k$ of the matrix $A_{k,n}$, we have the following result.

\begin{rem}\label{decomp}
The matrix $A_{k,n}$ can be decomposed as $A_{k,n}= P+2I_n$, where $P=\mathbf{1}\cdot v^{T}$ is an $n\times n$ matrix with $\mathbf{1}^{T}=(1,1,\cdots,1)_{n\times1}$, $v^{T}=(v_1,\ldots,v_n)_{1\times n}$, with $v_i=F_{k,2i}L_{k,2i-1}$, and $I_n$ denotes the identity matrix of size $n$.
\end{rem}

\begin{lem}\label{P}
For every $k\in \mathbb{N}$, we have $P^k=(v^{T}\mathbf{1})^{k-1}P$.
\end{lem}

\begin{proof}
We proceed by induction on $k$. For $k=2$, we have
\[
P^2=P\cdot P=(\mathbf{1}\cdot v^{T})(\mathbf{1}\cdot v^{T})
= \mathbf{1}\cdot (v^{T}\cdot\mathbf{1})\cdot v^{T}
=(v^{T}\cdot\mathbf{1})P.
\]
Here, note that $v^{T}\cdot\mathbf{1}=\displaystyle \sum_{j=1}^{n}v_{j}$. Thus, the base case holds.
\par\bigskip

\noindent Now, we assume that the identity holds for some fixed $k\in \mathbb{N}$, that is,
\[
P^{k}=(v^{T}\cdot \mathbf{1})^{k-1}P.
\]

\noindent We want to show that $P^{k+1}=(v^{T}\cdot \mathbf{1})^{k}P$. Indeed,
\begin{align*}
    P^{k+1} &= P^{k}\cdot P \\
    &= (v^{T}\cdot \mathbf{1})^{k-1}P\cdot P \\
    &= (v^{T}\cdot \mathbf{1})^{k-1}P^2 \\
    &= (v^{T}\cdot \mathbf{1})^{k-1} (v^{T}\cdot \mathbf{1})\cdot P \\
    &= (v^{T}\cdot \mathbf{1})^{k}\cdot P.
\end{align*}
This completes the inductive step, and thus the proof.
\end{proof}

\noindent For the power $k$ of the matrix $A_n$ we have the following result.

\begin{thm}
Let $A_{k,n}$ the matrix defined as before then for any $m\in \mathbb{N}$, 

\[A_{k,n}^{m}=2^{m}I_{n}\frac{\lambda_{2}^{m}-2^{m}}{\lambda_{2}-2}\cdot \mathbf{1}\cdot v^{T}\]

\noindent where $I_n$ denote the identity matrix of size $n$, $\lambda_{2}$ is the eigenvalue of Theorem \ref{eigegen}, $\mathbf{1}^{T}=(1,1,\cdots,1)$ and $v^{T}=(F_{k,2}L_{k,1},F_{k,4}L_{k,3},\cdots,F_{k,2n}L_{k,2n-1})$. 
\end{thm}

\begin{proof}
According to Remark \ref{decomp}  and considering that $(\mathbf{1}\cdot v^{T})\cdot(2I_n)=(2I_n)\cdot (\mathbf{1}\cdot v^{T})$ we apply the Binomial theorem, that is: 
\begin{align*}
A_{k,n}^{m}&= (1\cdot v^{T}+2I_n)^m\\
         &=\displaystyle \sum_{k=0}^{m}\binom{m}{k}(\mathbf{1}\cdot v^{T})^{k}(2I_{n})^{m-k}\\
         &=\displaystyle \sum_{k=0}^{m}\binom{m}{k}(\mathbf{1}\cdot v^{T})^{k}\cdot 2^{m-k}I_{n}\\
         &=2^{m}I_{n}+\displaystyle \sum_{k=1}^{m}\binom{m}{k}(\mathbf{1}\cdot v^{T})^k\cdot 2^{m-k}I_{n}\\
           &=2^{m}I_{n}+\displaystyle \sum_{k=1}^{m}\binom{m}{k}P^k\cdot 2^{m-k}I_{n}        
\end{align*}
Applying Lemma \ref{P} we have, 
\begin{align}\label{3.3}
    A_{n}^{m}=2^mI_{n}+\underbrace{\Big( \displaystyle \sum_{k=1}^{m}\binom{m}{k}2^{m-k}\cdot(v^{T}\cdot \mathbf{1})^{k-1}\Big)}_{(1)}\cdot P
\end{align}

\noindent In order to reduce $(1)$, we take $x=v^{T}\cdot \mathbf{1}$. Thus, $(1)$ can be written as 

\begin{align*}
S&=\displaystyle \sum_{k=1}^{m} \binom{m}{k}2^{m-k}x^{k-1}\\
xS&=\displaystyle \sum_{k=1}^{m} \binom{m}{k}2^{m-k}x^{k}\\
xS&=\displaystyle \sum_{k=0}^{m} \binom{m}{k}2^{m-k}x^{k}-2^{m}\\
xS&=(2+x)^{m}-2^{m} \\
S&=\frac{(2+x)^{m}-2^{m}}{x}
\end{align*}
\noindent  So, \ref{3.3} is exactly, 
\[A_{k,n}^{m}=2^{m}I_{n}+\frac{(2+v^{T}\cdot \mathbf{1})^m-2^{m}}{v^{T}\cdot \mathbf{1}}\cdot \mathbf{1}\cdot v^{T}\]
\noindent But considering  that $v^{T}\cdot \mathbf{1}=\displaystyle \sum_{j=1}^{n}v_j=\sum_{j=1}^{n}F_{k,2j}L_{k,2j-1}$ by Theorem \ref{sumgen} and Theorem \ref{eigegen} We are done.
\end{proof}

\normalcolor 

\subsection{Inverse of $A_{k,n}$}\label{inversegen}
For the inverse the following result holds 

\begin{thm}\label{inverse}
Let $A_{k,n}$ defined as before then $A_{k,n}$ is invertible matrix and its inverse is given by

\[
A_{k,n}^{-1}=
\frac{1}{2} I_n -\frac{1}{2(\lambda_{2})}\,\mathbf{1}v^T.
\]

\noindent where $\lambda_2$ is given in Theorem \ref{eigegen}.
\end{thm}

\begin{proof}
By Remark \ref{decomp} We have that $A_{k,n}=2(I_{n}+\frac{1}{2}\mathbf{1}\cdot v^{T})$ so $A_{k,n}$ is invertible if and only if $I_{n}+\frac{1}{2}\mathbf{1}\cdot v^{T}$ is invertible , that is, 

\[A_{k,n}^{-1}=\frac{1}{2}\Bigl(I_{n}+ \frac{1}{2}\cdot \mathbf{1}v^{T}\Bigr)^{-1}\]

\noindent We fix $u=\frac{1}{2}\cdot \mathbf{1}$ and $w=v$ then $I_{n}+\frac{1}{2}\cdot \mathbf{1}\cdot v^{T}=I_{n}+uw^{T}$. By Sherman - Morrison formula (see \cite{GolubVanLoan}, page 65) the matrix $I_{n}+uw^{T}$ is invertible if $\mathbf{1}+w^{T}\cdot u\neq 0$ and its inverse is given by 

\[(I_{n}+uw^{T})^{-1}=I_{n}- \frac{uw^{T}}{1+w^{T}u}\]

\noindent  In our context, 

\[w^{T}u=v^{T}\cdot\Big(\frac{1}{2}\cdot \mathbf{1}\Big)=\frac{1}{2}v^{T}\cdot \mathbf{1}\]

\noindent  Since $v^{T}\cdot \mathbf{1}>0$ it follows that $\mathbf{1}+w^{T}u=\mathbf{1}+\frac{1}{2}v^{T}\cdot \mathbf{1}\neq 0$ and therefore the inverse exists. Applying the Sherman- Morrison formula we have 

\[\Bigl(I_{n}+\frac{1}{2} \mathbf{1}\cdot v^{T}\Bigr)^{-1}= I_{n}-\frac{\frac{1}{2}\cdot \mathbf{1}\cdot v^{T}}{1+\frac{1}{2}v^{T}\cdot \mathbf{1}}\]

\noindent Multiplying by $\frac{1}{2}$ we obtain 
\[A_{k,n}^{-1}=\frac{1}{2}I_{n}-\frac{1}{4}\cdot \frac{\mathbf{1}\cdot v^{T}}{1+\frac{1}{2}v^{T}\cdot \mathbf{1}}\]

\noindent Note that 

\[ 1+\frac{1}{2}v^{T}\cdot\mathbf{1}=\frac{2+v^{T}\cdot \mathbf{1}}{2}\]

\noindent We simplify the expression to obtain 

\[A_{k,n}^{-1}=\frac{1}{2}I_{n}-\frac{1}{2(2+v^{T}\cdot \mathbf{1})} \mathbf{1}\ v^{T}\]

\noindent which according with Theorem \ref{eigegen} completes the proof.
\end{proof}

\subsection*{Acknowledgment}
\noindent The first author expresses his deep gratitude to his family and his wife for their constant support, patience, and encouragement throughout the development of this work. He also wishes to dedicate this paper to ITENUA (Semillero de Investigación en Teoría de Números y Álgebra) on the occasion of its 15th birthday.\par \bigskip 

\noindent The second and the tird author acknowledges the institutional support Universidad Pedagógica y Tecnológica de Colombia.
\par \bigskip 
\paragraph{Author contributions} P. F. F. Espinosa , M.L.A. Torres and C.A.A. Cadena contributed equally to the manuscript.
\paragraph{Data Availability} No datasets were generated or analysed during the current study.
\subsection*{Declarations}
\paragraph{Competing interests} The authors declare no competing interests.

\end{document}